\documentclass[11pt]{amsart}

\usepackage{amsmath,amsfonts,amsthm,amssymb}
\usepackage{latexsym}

\newtheorem{defn}{Definition}[section]

\newtheorem{thm}{Theorem}
\newtheorem{lem}[defn]{Lemma}
\newtheorem{prop}[defn]{Proposition}

\newtheorem{rem}[defn]{Remark}

\usepackage{bbm}
\usepackage{amsthm}
\usepackage{pst-all}

\usepackage{latexsym}
\usepackage{dsfont}
\usepackage[dvips]{epsfig}
\usepackage[latin1]{inputenc}
\usepackage{amssymb}
\usepackage{appendix}
\usepackage{comment}

\usepackage{graphicx}
\usepackage{epsfig}

\begin{document}

\title[Large Deviations for Weighted Sums]{Large Deviations for Weighted Sums of Stretched Exponential Random Variables\thanks{Supported
    by NSF CMMI-1234100 and ARO W911NF-12-1-0222}} 

%\DEDICATORY{Dedicated to the memory of ...} % Optional

%%%%%%%%%%%%%%%%%%%%%%%%%%%%%%%%%%%%%%%%%%%%%%%%%%%%%%%%%%%%%%%%%%%
%%                                                               %%
%% Authors (please edit and customize):                          %%
%%                                                               %%
%%%%%%%%%%%%%%%%%%%%%%%%%%%%%%%%%%%%%%%%%%%%%%%%%%%%%%%%%%%%%%%%%%%

\setcounter{footnote}{3}

\author{Nina Gantert}
 \address{Technische Universit\"at M\"unchen, Germany} 
\email{gantert@ma.tum.de}

\author{Kavita Ramanan}
 \address{Brown University, USA} 
  \email{Kavita\_Ramanan@brown.edu}

\author{Franz Rembart} 
\address{University of Oxford, UK}
\email{franz.rembart@stats.ox.ac.uk}

% University of Washington, USA.
%    \EMAIL{burdzy@math.washington.edu}}
%  \and %% remove this line and below if single author
%  Djalil~Chafa\"{\i}\footnote{Universit\'e Paris-Est Marne-la-Vall\'ee,
%    France. \BEMAIL{djalil@chafai.net} \url{http://djalil.chafai.net/}}}%AUTHORS
%% Type \and between all consecutive authors (not only before the last author).
%% Note: you may use \BEMAIL to force a line break before e-mail display.

%% Here is a compact example with two authors with same affiliation
%% \AUTHORS{%
%%  Michael~First\footnote{Some University. \EMAIL{mf,js@uni.edu}
%%  \and 
%%  John~Second\footnotemark[2]}%AUTHORS
%% Note: The \footnotemark is the footnote number that you wish to reuse. Here
%% it is [2] (we took into account the footnote generated by \thanks in title).

%%%%%%%%%%%%%%%%%%%%%%%%%%%%%%%%%%%%%%%%%%%%%%%%%%%%%%%%%%%%%%%%%%%
%%                                                               %%
%% Please edit and customize the following items:                %%
%%                                                               %%
%%%%%%%%%%%%%%%%%%%%%%%%%%%%%%%%%%%%%%%%%%%%%%%%%%%%%%%%%%%%%%%%%%%

% Separate items with ;

 % Edit. Separate items with ;
%\AMSSUBJSECONDARY{FIXME:} % Optional, separate items with ;

%\SUBMITTED{} % Edit.
%\ACCEPTED{} % Edit.

%%%%%%%%%%%%%%%%%%%%%%%%%%%%%%%%%%%%%%%%%%%%%%%%%%%%%%%%%%%%%%%%%%%
%%                                                               %%
%% Please uncomment and edit if you have an arXiv ID:            %%
%%                                                               %%
%%%%%%%%%%%%%%%%%%%%%%%%%%%%%%%%%%%%%%%%%%%%%%%%%%%%%%%%%%%%%%%%%%%

%\ARXIVID{NNNN.NNNNvn} % Edit.
%\ARXIVPASSWORD{xxxxx} % Edit.
%\HALID{hal-NNN} % Edit.

%%%%%%%%%%%%%%%%%%%%%%%%%%%%%%%%%%%%%%%%%%%%%%%%%%%%%%%%%%%%%%%%%%%
%%                                                               %%
%% The following items will be set by the Managing Editor.       %%
%%                                                               %%
%%%%%%%%%%%%%%%%%%%%%%%%%%%%%%%%%%%%%%%%%%%%%%%%%%%%%%%%%%%%%%%%%%%

%\VOLUME{0}
%\YEAR{2014}
%\PAPERNUM{0}
%\DOI{vVOL-PID}

%%%%%%%%%%%%%%%%%%%%%%%%%%%%%%%%%%%%%%%%%%%%%%%%%%%%%%%%%%%%%%%%%%%
%%                                                               %%
%% Please edit and customize the abstract:                       %%
%%                                                               %%
%%%%%%%%%%%%%%%%%%%%%%%%%%%%%%%%%%%%%%%%%%%%%%%%%%%%%%%%%%%%%%%%%%%

\begin{abstract}
We consider the probability that a weighted sum of $n$ 
  i.i.d. random variables $X_j, j = 1, \ldots, n$, with stretched exponential tails is
  larger than its expectation and   determine the rate of its decay, 
  under suitable conditions on the weights. We show that the decay is
  subexponential, and  identify the rate function in terms of the
   tails of $X_j$ and the weights. 
Our result generalizes the large deviation principle given by Kiesel
and Stadtm\"{u}ller 
\cite{KieSta00} as well as the tail asymptotics for sums of i.i.d. random 
variables provided by Nagaev \cite{Nag69,Nag79}.  
As an application of our result, motivated by random projections of
high-dimensional vectors, we consider
  the case of random, self-normalized weights that are independent   
  of the sequence $\{X_j\}_{j \in \mathbb N}$, identify the decay rate for both the quenched and annealed large
  deviations in this case, and show that they coincide. As another
  example we consider weights derived from kernel functions that arise
  in non-parametric regression. 
\end{abstract}

\maketitle

%%%%%%%%%%%%%%%%%%%%%%%%%%%%%%%%%%%%%%%%%%%%%%%%%%%%%%%%%%%%%%%%%%%
%%                                                               %%
%% Please add your own macros and environments below:            %%
%%                                                               %%
%% If possible, avoid using \def and use instead \newcommand     %%
%% If possible, avoid defining your own environments, and use    %%
%% instead the environments already defined by ejpecp:           %%
%%  assumption, assumptions, claim, condition, conjecture,       %%
%%  corollary, definition, definitions, example, exercise, fact, %%
%%  facts, heuristics, hypothesis, hypotheses, lemma, notation,  %%
%%  notations, problem, proposition, remark, theorem             %%
%%                                                               %%
%%%%%%%%%%%%%%%%%%%%%%%%%%%%%%%%%%%%%%%%%%%%%%%%%%%%%%%%%%%%%%%%%%%

%\newcommand{\ABS}[1]{\left(#1\right)} % example of author macro
%\newcommand{\veps}{\varepsilon} % another example of author macro

\setlength{\textheight}{22.25cm} \setlength{\textwidth}{17.25cm} \setlength{\oddsidemargin}{-20pt} \setlength{\evensidemargin}{-20pt}
\setlength{\topmargin}{-41pt}

\setlength{\unitlength}{1cm} \setlength{\parskip}{3pt} \setlength{\parindent}{0pt}

%\font \mrm=cmr8 scaled \magstep 0 \DeclareMathAlphabet{\mathbfit}{OT1}{cmr}{bx}{it} \DeclareRobustCommand{\vec}[1]{\mathbfit{#1}}

\newcommand{\erw}[1]{\mbox{E} [#1] }
\renewcommand{\baselinestretch}{1.1}

\newcommand{\bea}{\begin{eqnarray*}}
\newcommand{\eea}{\end{eqnarray*}}
\newcommand{\be}{\begin{eqnarray}}
\newcommand{\ee}{\end{eqnarray}}
\newcommand{\beq}{\begin{equation}}
\newcommand{\eeq}{\end{equation}}
\newcommand{\schw}{\stackrel{\mathcal{D}}{\longrightarrow}}
\newcommand{\stab}{\stackrel{\mathcal{D}_{st}}{\longrightarrow}}
\newcommand{\gstab}{\stackrel{\mathcal{G}-\mathcal{D}_{st}}{\longrightarrow}}
\newcommand{\fstab}{\stackrel{\mathcal{F}-\mathcal{D}_{st}}{\longrightarrow}}
\newcommand{\zd}{\mathbb{Z}^d}

\newtheorem{assumptionletter}{Assumption}
\renewcommand{\theassumptionletter}{\Alph{assumptionletter}}
\def\ind{{1\!\!1}}

%%%%%%%%%%%%%%%%%%%%%%%%%%%%%%%%%%%%%%%%%%%%%%%%%%%%%%%%%%%%%%%%%%%
%%                                                               %%
%% No macro definitions below this line please!                  %%
%%                                                               %%
%%%%%%%%%%%%%%%%%%%%%%%%%%%%%%%%%%%%%%%%%%%%%%%%%%%%%%%%%%%%%%%%%%%

%\begin{document}

%%%%%%%%%%%%%%%%%%%%%%%%%%%%%%%%%%%%%%%%%%%%%%%%%%%%%%%%%%%%%%%%%%%
%%                                                               %%
%% No need for \maketitle.                                       %%
%%                                                               %%
%%%%%%%%%%%%%%%%%%%%%%%%%%%%%%%%%%%%%%%%%%%%%%%%%%%%%%%%%%%%%%%%%%%

%%%%%%%%%%%%%%%%%%%%%%%%%%%%%%%%%%%%%%%%%%%%%%%%%%%%%%%%%%%%%%%%%%%
%%                                                               %%
%% Please replace what follows by the body of your article       %%
%% (up to the bibliography):                                     %%
%%                                                               %%
%%%%%%%%%%%%%%%%%%%%%%%%%%%%%%%%%%%%%%%%%%%%%%%%%%%%%%%%%%%%%%%%%%%

\section{Introduction} \label{motivation}
Let $\{X_j\}_{j \in \mathbb N}$ be a sequence of independent and
identically distributed (i.i.d.)  random variables on a probability space $(\Omega, \mathcal F, \mathbb{P})$ with values in 
$\mathbb R$ and with finite expectation $m:=\mathbb{E}[X_1] < \infty$.
For $n \in \mathbb N$, let $S_n:= \sum_{j=1}^n X_j$, denote the partial sum and $\bar{S}_n :=
S_n/n$ the empirical mean values. 
The strong law of large numbers implies that $\bar{S}_n \to m$ almost surely.
Cram\'er's Theorem on large deviations tells us that, if the $X_j$
have finite exponential moments, that is, there exists $t > 0$ such
that 
\begin{equation} \label{expmom}
  M(t):=\mathbb{E}[\exp\left(tX_1\right)] < \infty, 
% \quad \hbox{ for  some }\, t > 0,
 \end{equation}
then for any $x > m$, the probability 
$\mathbb{P}\left(\bar{S}_n\geq x\right)$ decays exponentially. More precisely,
\begin{equation*} \lim_{n \rightarrow \infty} \frac{1}{n} \log
  \mathbb{P}\left(\bar{S}_n \geq x\right) = -\Lambda^{*}(x),  \end{equation*}
where 
$\Lambda^{*}(x):=\sup_{t \geq 0} \left\{ tx-\log M(t) \right\} > 0$. We will refer to this case as the ``light-tailed'' case.
It is well known that if $M(t) = + \infty$ for all $t > 0$, the probabilities
$P\left(\bar{S}_n \geq x\right)$ decay slower than exponentially. The reason is that, in contrast to when \eqref{expmom} holds, a ``deviation'' of the type
$\bar{S}_n \geq x$ is produced by the event that {\sl just one} of the random variables takes a large value.
For instance, if there is $r \in (0,1)$ and $c > 0$ such that
$P\left(X_1 \geq t\right) = c \exp(-t^r)$ for $t$ large enough, then
\begin{equation} \label{partcase}
\lim_{n \rightarrow \infty} \frac{1}{n^r} \log
\mathbb{P}\left(\bar{S}_n \geq x\right) = -(x-m)^r, \quad \forall x > m.  \end{equation}
The result in \eqref{partcase} goes back to \cite{Nag69} and it will also
follow from our main result, Theorem \ref{MR}.
Cram\'er's Theorem was generalized by \cite{KieSta00} to weighted sums of
i.i.d.\ random variables, see Section \ref{sec-lt} below for a precise statement of their results.
Our main result, Theorem \ref{MR}, gives a corresponding statement for
weighted sums of i.i.d.\ random variables with stretched exponential tails. 
One motivation to consider weighted sums, which is elaborated upon in 
Section \ref{Example1},  comes from random
projections of high-dimensional vectors, which are of relevance in 
asymptotic geometric analysis  \cite{DiaFre84,Mec12} and data analysis 
\cite{BinMan01}.   
 Another motivation stems from
statistics (kernel functions, moving averages) considered for the light-tailed case in \cite{KieSta00}, 
since stretched exponential random variables 
arise in many applications.
See Section \ref{Example2} for an example.

This article is organized as follows: We first present the result and
the regularity conditions from \cite{KieSta00} in Section \ref{sec-lt}.  Our
main result, Theorem \ref{MR}, is given in Section \ref{sec22}, and
its proof is presented in Section \ref{sec23}.  Finally, in Section
\ref{Example1}, we give an application to random weights,  and in
Section \ref{Example2}, we consider weights derived from kernel
functions that arise in non-parametric regression. 

\section{The Light-Tailed Case}
\label{sec-lt}

For $n \in  \mathbb N$, let $\{a_j(n)\}_{j \in \mathbb N}$ be a sequence of real numbers which we will call weights.
% where $n \in  \mathbb N$, such that for $\nu=1$ (and hence for all $\nu \in \mathbb N)$
%\begin{equation} \sum \limits_{j=1}^{\infty} |a_j(n)|^{\nu} < \infty \text{ for every } n \in \mathbb %N. \label{eq:n1} \end{equation}
For $n \in \mathbb N$ define the weighted sum
\begin{equation}\label{eq:n2}
 \bar{S}_n:=\sum \limits_{j=1}^{n} a_j(n) X_j  \end{equation}
and the measure $\mu_{n}$ on ${\mathcal B}\left(\mathbb R
\right)$, the set of Borel sets in $\mathbb{R}$,   as 
\begin{equation}\mu_{n}\left(A\right):=\mathbb{P}\left(\bar{S}_n \in
    A\right), \quad A \in {\mathcal B}(\mathbb{R}).  \label{eq22}\end{equation} 
%Note that $\bar{S}_n$ can be defined almost surely since we assume that $\mathbb{E}[X_1^2]< \infty$ 
%and $\sum \limits_{j=1}^{\infty} a_j(n)^2<\infty$. 
When the $\{X_j\}_{j \in \mathbb N}$ have finite exponential moments,
that is the moment generating 
function $M(t)$ defined in \eqref{expmom} is finite for all $t \in \mathbb{R}$, 
a large deviation principle  for the sequence of weighted sums
$\{\bar{S}_n\}_{n \in \mathbb N}$ was established  in \cite{KieSta00} 
under suitable assumptions on the weights, see Assumption \ref{as_A}
below.   The ``classical'' case of Cram\'{e}r's theorem corresponds to 
$a_j(n) = 1/n, j= 1, 2,  \ldots ,n$, $n \in \mathbb N$.  
%the cumulant moment-generating function of $X_1$  satisfies
\begin{assumptionletter}
         \label{as_A}
         \renewcommand{\theenumi}{(A.\arabic{enumi})}
         \renewcommand{\labelenumi}{\theenumi}
\begin{enumerate}
\item \label{as_A1} 
There exists a sequence of real numbers $\{s_{\nu}\}_{\nu \in \mathbb
  N}$ such that $s_{\nu}\not=0$ for all $\nu \in \mathbb N$, the limit
$s:=\lim\limits_{\nu \rightarrow \infty} \sqrt[\nu]{|s_{\nu}|}$ exists
and 
\begin{equation} \sum \limits_{j=1}^{n} {a_j(n)}^{\nu}=\frac{s_{\nu}}{n^{\nu-1}}R(\nu, n) \text{ for all }\nu \text{ and }n \in \mathbb N, \label{eq:e3} \end{equation} 
for some function $R: {\mathbb N}^2 \rightarrow \mathbb R$ that
satisfies,  for every $\nu \in \mathbb N$, $R\left(\nu, n \right) \rightarrow 1$ as $n \rightarrow \infty$.
\item \label{as_A2} 
There exist sequences  $\{r_{\nu}\}_{\nu \in \mathbb N}$ and
$\{\delta_n\}_{n \in \mathbb N}$ such that 
$\limsup_{\nu \rightarrow \infty} \sqrt[\nu]{r_\nu} \leq 1$,
$\lim_{n \rightarrow \infty} \delta_n = 0$ and the error term satisfies
\begin{equation} |R(\nu, n) - 1| \leq r_{\nu} \frac{(1+\delta_n)^{\nu}}{n}\text{ for all }\nu\text{ and }n. \label{eq:e4} \end{equation}
\end{enumerate}
\end{assumptionletter}

Now, let $\Lambda$ denote the cumulant (or log moment) generating
function of $X_1$, and let $\{c_{\nu}\}_{\nu \in \mathbb N}$ be the sequence of coefficients that
arise in the power series expansion for $\Lambda$: 
\begin{equation} \Lambda(t):=\log M(t) =  \sum \limits_{\nu=1}^{\infty} \frac{c_{\nu}}{\nu!}t^{\nu}, \quad t \in \mathbb{R}.
  \hspace{4mm} \label{eq:n3b}\end{equation}  
Also, for $t > 0$, let $\chi(t):=\sum \limits_{\nu=1}^{\infty}
\frac{s_{\nu}c_{\nu}}{ \nu!}t^{\nu}$, and let $\chi^{*}$ denote its 
Legendre-Fenchel transform:  
\begin{equation} \chi^{*}(t):=\sup \limits_{t \in \mathbb R}
  \{tx-\chi(t)\}.\label{eq:n4} \end{equation}
It was shown in \cite{KieSta00} that under Assumption \ref{as_A} the
sequence of measures $\{\mu_{n}\}_{n \in \mathbb N}$ on
${\mathcal B} (\mathbb R)$ defined in
\eqref{eq22} satisfies a large deviation principle with speed $n$ and
rate function $\chi^{*}$.   Recall that this means that 
\[  - \inf_{x \in A^\circ} \chi^* (x) \leq  \liminf_{n  \rightarrow \infty}
\frac{1}{n} \mu_n (A^\circ) \leq \limsup_{n \rightarrow \infty}
\frac{1}{n}\mu_n (\bar{A}) \leq - \inf_{x \in \bar{A}} \chi^* (x),    \quad
\forall A
\in {\mathcal B} (\mathbb R),  \]
where $A^\circ$ and $\bar{A}$, respectively, represent the interior
and the closure of the set $A$.  

%More precisely, the following holds.

%\begin{thm}[Large Deviations for Weighted Sums, Light Tails] \label{kiesel}
%Let $(X_j)$ be a sequence of i.i.d. random variables  satisfying \eqref{eq:n3} and let $\bar{S}_n$ be defined by 
%\eqref{eq:n2} with weights satisfying Assumption \ref{as_A}. Then the sequence of measures $(\mu_n)$ satisfies a 
%large deviation principle with speed $n$ and rate function $\chi^{*}$, i.e. we have
%\renewcommand{\labelenumi}{(\roman{enumi})}
%\begin{enumerate}
%\item for any closed set $F\subset \mathbb R$,
%\begin{equation*} \limsup \limits_{n \rightarrow \infty} \frac{1}{n}\log \mu_{n}(F) \leq - \inf \limits_{x \in F} 
%\chi^{*}(x), \label{UB}\end{equation*}
%\item for any open set $G \subset \mathbb R$,
%\begin{equation*} \liminf \limits_{n \rightarrow \infty} \frac{1}{n}\log \mu_{n}(G) \geq - 
%\inf \limits_{x \in G} \chi^{*}(x). \label{LB} \end{equation*}
%\end{enumerate}
%\end{thm}
%\begin{proof} See \cite{KieSta00}. 
%\end{proof}

\begin{rem}
{\em 
In fact, \cite{KieSta00} provides a more general result 
that considers an
infinite sum and refers to a general scale within the regularity
conditions (cf. Assumption \ref{as_A}), that is, they prove large
deviations for the family of weighted sums of the form
$A(\lambda):=\sum_{j=1}^{\infty} a_j(\lambda)X_j$,
where $\lambda \in I$ and either $I=\mathbb N$ or
$I=[0,\infty]$.} \end{rem}

Our goal will be to relax the finiteness assumption \eqref{eq:n3b} on
the moment generating function $M(\cdot)$.   

\section{Main Result} \label{sec22}

In order to present our large deviation result for weighted sums of stretched exponential random variables, we will use slightly different
assumptions on the weights from those used in \cite{KieSta00}. We will
restrict our considerations to non-negative weights.
As we show in  Lemma
\ref{Bisbetter} below,  in this case, our assumptions are weaker than those  used in
 \cite{KieSta00}.

\begin{assumptionletter}
         \label{as_B}
         \renewcommand{\theenumi}{(B.\arabic{enumi})}
         \renewcommand{\labelenumi}{\theenumi}
\begin{enumerate}
\item \label{as_B1} There exists a real number $s_1\neq 0$ 
such that the sequence $\{R(1,n)\}_{n \in \mathbb N}$ of real numbers defined
by 
\begin{equation*} \sum \limits_{j=1}^{n}
  a_j(n)  = s_1 R(1, n), \text{ for all } n \in \mathbb
  N,  \label{eq:f3} \end{equation*} satisfies $R(1,n) \rightarrow 1$ as $n \rightarrow \infty$.\\
\item \label{as_B2} 
There exists a real number $s$ such that for $a_{max}(n):=\max_{1 \leq
  j \leq n} a_j(n)$, 
\begin{equation} \lim \limits_{n \rightarrow \infty} n \cdot a_{max}(n) = s. \label{eq:f4} \end{equation} 
\end{enumerate}
\end{assumptionletter}
Examples for weight sequences that satisfy both Assumption \ref{as_A} and Assumption \ref{as_B} include Valiron means, see
\cite{KieSta00} as well as kernel functions (see Section
\ref{kernels}). 

%We study the tail behavior of $\bar{S}_n$ in the case when $(X_j)$ have stretched exponential tails, and hence fail to d
%have finite exponential moments. To this end, we 
%recall the definition of slowly varying functions and their most important properties (see \cite{9} and \cite{4}).

Recall that
a function $\ell: (0,  \infty) \rightarrow (0, \infty)$ is called \textbf{slowly varying} (at infinity) if for every $a >0$,
\begin{equation} \lim \limits_{x \rightarrow \infty} \frac{\ell(ax)}{\ell(x)}=1. \label{sv1} \end{equation}

We now state our main result. 

\begin{thm}[Large Deviations for Weighted Sums, Stretched Exponential Tails] \label{MR}
Let $\{X_j\}_{j \in \mathbb N}$ be a sequence of i.i.d. random variables on a probability space $(\Omega, \mathcal F, \mathbb P)$ with 
\begin{equation}\label{allmom}
\mathbb E[X_1^k]< \infty \quad \forall k \in \mathbb{N}, 
\end{equation}
and let $m:=\mathbb E[X_1]$.
%and $\sigma^2:=Var[X_1]$. 
Suppose that there exist a constant $r \in (0,1)$ and slowly varying functions $b$, $c_1$, $c_2:  (0, \infty) \rightarrow (0, \infty)$ and a constant $t^{*} >0$ such that for $t \geq t^{*}$,
\begin{equation} c_1(t) \exp\left(-b(t)t^r\right) \leq \mathbb{P}\left(X_1 \geq t\right) \leq c_2(t) \exp\left(-b(t)t^r\right). \label{eq:e1} \end{equation}
For every $n \in \mathbb N$, let $\{a_j(n)\}_{j \in \mathbb
  N}$ be a sequence of non-negative numbers that satisfy Assumption
\ref{as_B} with associated constants  $s_1, s \in \mathbb R$,
and let $\{\bar{S}_n\}_{n \in \mathbb N}$ be the sequence of weighted sums defined in
\eqref{eq:n2}.  Then 
\begin{equation}\lim \limits_{n \rightarrow \infty} \frac{1}{b(n)n^{r}} \log \mathbb{P} \left(\bar{S}_n \geq x\right) =
 -\left(\frac{x}{s}-\frac{s_1}{s}m\right)^r, \quad \forall x > s_1m.  \label{eq:e6}\end{equation}
\end{thm}

\begin{rem}
\label{rem-mt}
{\em The non-negativity  assumption on the weights could be relaxed
    only if one had more information about the lower tail of  the $\{X_j\}$,
    that is,  about the probabilites $\mathbb P(X_1 \leq -t)$ for $t  > 0$. Consider
    the following example: $a_j(n) = 1/n, j= 1, \ldots, \lfloor 2n/3 \rfloor$, $a_j(n)
    = - 1/n, j= \lfloor 2n/3 \rfloor + 1, \ldots , n$ (where, for $z
    \in \mathbb{R}$, 
    $\lfloor z \rfloor$ represents the greatest integer less than or
    equal to $z$). Then Assumption B is satisfied with $s_1 = 1/3$ and
    $s=1$. Take  i.i.d.\ random variables $\{X_j\}_{j \in \mathbb N}$
    with mean $m$ 
    that 
    satisfy \eqref{allmom} and \eqref{eq:e1}  and, in 
    addition, satisfy $\mathbb{P}(X_1
    \leq -t) = \exp ( -t^\alpha)$ for  some $\alpha$ with $0 <  \alpha  <
    r$, and $t$ large enough. Then,  
for every $x >m/3$, it can be shown that }
%there exists a constant $\gamma_x = (x - m/3)^\alpha > 0$ such that 
\begin{equation} \label{differdecay}
\lim \limits_{n \rightarrow \infty} \frac{1}{n^{\alpha}} \log \mathbb{P} \left(\bar{S}_n \geq x\right) =
 - \left( x  - \frac{m}{3} \right)^\alpha. 
\end{equation}
{\em Indeed, to show  \eqref{differdecay},  for any $\varepsilon > 0$, first write
 \[ \mathbb{P} (\bar{S}_n \geq x) \leq
 \mathbb{P}\left(\sum_{i=1}^{\lfloor 2n/3\rfloor} X_i \geq 2n(m+ \varepsilon)/3\right)
+\mathbb{P}\left(\sum_{i=\lfloor 2n/3 \rfloor+ 1}^n  (-X_i)  \geq n(x-
2(m+ \varepsilon)/3)\right). \]
Then, applying Theorem \ref{MR} twice, first to $\{X_j\}_{j \in
  \mathbb N}$ and then to $\{-X_j\}_{j \in \mathbb N}$, both times
with $a_j(n) = 1/n, j \in \mathbb N,$ and recalling that $\alpha < r$,
we  infer that  as $n \rightarrow \infty$, 
$n^{-\alpha} \ln \mathbb{P}(\sum_{i=1}^{\lfloor 2n/3\rfloor} X_i \geq 2n(m+
\varepsilon)/3) = -\infty$ 
%is asymptotically negligible with respect to 
%$n^{-\alpha} \ln \mathbb{P}(\sum_{i=\lfloor 2n/3 \rfloor+ 1}^n  (-X_i)  \geq n(x-
%2((m+ \varepsilon)/3))$ 
%(because $\alpha < r$),
 and hence, 
\begin{eqnarray*}
 \limsup_{n \rightarrow \infty}  \frac{1}{n^\alpha} \log \mathbb{P} \left(
\bar{S}_n \geq x\right)  & \leq & \lim_{n \rightarrow \infty}  \frac{1}{n^\alpha} \log \mathbb{P} \left(
\sum_{i=\lfloor 2n/3 \rfloor+ 1}^n  (-X_i) 
\geq 
n(x-2((m+\varepsilon)/3)\right) \\
& = & - \left( x- \frac{m-2\varepsilon }{3} \right)^\alpha. 
\end{eqnarray*}
Sending $\varepsilon \to 0$, we see that \eqref{differdecay}
holds  with $\leq$ instead of equality.  To show the opposite inequality in \eqref{differdecay}, write 
\[  \mathbb{P} (\bar{S}_n \geq x) \geq \mathbb{P} \left(
  \sum_{i=1}^{\lfloor 2n/3\rfloor} X_i  \geq
  n(2m/3 - \varepsilon)\right) \cdot \mathbb{P}\left(\sum_{i=\lfloor
    2n/3 \rfloor+ 1}^n (-X_i) \geq n(x- 2m/3 + \varepsilon \right). \]
The first probability on the right-hand side goes to $1$ due to the law of large
numbers.  Once again, applying Theorem \ref{MR} to $\{-X_j\}_{j \in
  \mathbb N}$  with $a_j(n) = 1/n, j \in \mathbb N$, for the second
term on the right-hand side, and then letting $\varepsilon \to 0$, we obtain \eqref{differdecay}
with $\geq$ instead of equality. 
Together,  both inequalities  prove \eqref{differdecay}.
However, we cannot recover $\alpha$ from the assumptions in Theorem \ref{MR}.}
\end{rem}

\begin{rem}{\em For the same reason as in the last remark, namely
    that the only assumption on the lower tail of  $\{X_j\}_{j \in
      \mathbb{N}}$ is \eqref{allmom}, we cannot strenghten
    \eqref{eq:e1} to a large deviation principle without imposing
    further assumptions. For  $x < s_1m$, the
    decay of $\mathbb{P}(\bar{S}_n \leq x)$ is determined by the lower
    tail of the $\{X_j\}$.  For example, if the $\{X_j\}_{j \in \mathbb{N}}$ are bounded below, Cram\'er's
Theorem implies that
$\mathbb{P}(\bar{S}_n \leq x)$ decays exponentially in $n$. If, on the 
other hand, $\mathbb{P}(X_1 \leq -t) = \exp (-t^\alpha)$ with  $0 < \alpha < r$,
then as in Remark \ref{rem-mt} we can show 
$-\infty < \lim_{ n \to \infty} n^{-\alpha} \log \mathbb{P}(\bar{S}_n \leq x) < 0$.}
\end{rem}

Stretched exponential distributions have been proposed as a complement  
 to the frequently used power law distributions to model many
 naturally occurring heavy-tailed distributions.   
Any distribution that satisfies \eqref{eq:e1} and is bounded below 
also satisfies \eqref{allmom}.  A concrete example is the 
Weibull distribution with shape parameter lying in the interval
$(0,1)$. 
Before proceeding to the proof of Theorem \ref{MR}, let us comment on
the relationship between Assumptions A and B. In fact, for a non-negative sequence of weights, Assumption \ref{as_B} is weaker than Assumption \ref{as_A}. see Lemma \ref{Bisbetter}. To see that it is strictly weaker, consider the sequence of weights defined by $a_j(n)=n^{-1}+n^{-(1+\varepsilon)}$, $j=1,...,n$, for some $\varepsilon \in (0, \frac{1}{2})$, for which it is easy to show that Assumption \ref{as_B} holds, but \ref{as_A2} cannot be satisfied.

\begin{lem}[Relationship between Assumptions \ref{as_A} and \ref{as_B}]\label{Bisbetter}
For every $n \in \mathbb N$, let $\{a_j(n)\}_{j \in \mathbb N}$ be a
sequence of non-negative real numbers that satisfy Assumption \ref{as_A}. Then Assumption \ref{as_B} holds.
\end{lem}
\begin{proof} Given weights $\{a_j(n)\}_{j \in \mathbb N}$ that
  satisfy Assumption \ref{as_A},  clearly \ref{as_B1} follows immediately from
  \ref{as_A1}. 
It only remains to show \ref{as_B2}.  First, note that by Assumption \ref{as_A2}, $R(\nu, n)$ satisfies the inequality
\begin{equation} 1 - r_\nu\frac{(1+\delta_n)^\nu}{n} \leq R(\nu, n) \leq 1 + r_\nu\frac{(1+\delta_n)^\nu}{n}.   \label{inequalRR} \end{equation}
Moreover, for any $\varepsilon > 0$, we can find $\nu^*
(\varepsilon) \in \mathbb N$ and  $n^* (\varepsilon) \in \mathbb N$ such that  
\begin{equation} 0 \leq r_\nu \leq (1+\varepsilon)^\nu, \quad  \forall
  \nu \geq \nu^*(\varepsilon), \quad \mbox{ and } \quad 0 \leq \delta_n \leq
  \varepsilon, \quad \forall  n \geq n^* (\varepsilon ).\label{inequr} \end{equation}
%Furthermore, there exists $n^* \in \mathbb N$ such that 
%\begin{equation} 0 \leq \delta_n \leq \varepsilon, \quad \forall  n\geq n^*.  \label{inequdelta} \end{equation}
By using the inequality $a_{max}(n)^{\nu} \leq \sum \limits_{j=1}^{ n
}a_j(n)^{\nu}$, \ref{as_A1} and \ref{as_A2} we see that for $\nu, n
\in \mathbb N$, 
\[ na_{max}(n) \leq n \left( \sum \limits_{j=1}^n
  a_j(n)^\nu\right)^\frac{1}{\nu} =  n (s_\nu R(\nu,
n))^\frac{1}{\nu}\cdot (n^{1-\nu})^\frac{1}{\nu}\leq
n^\frac{1}{\nu}(s_\nu)^\frac{1}{\nu} \left(  1 +
  r_\nu\frac{(1+\delta_n)^\nu}{n} \right)^\frac{1}{\nu}. 
\]
Together with \eqref{inequr}, this implies that
for $\varepsilon  > 0$, and $\nu \geq \nu^*(\varepsilon)$, $n \geq n^*(\varepsilon)$,  
\[ na_{max}(n) \leq (s_\nu)^\frac{1}{\nu} \left(n(1+\varepsilon)^{2\nu} + (1+\varepsilon)^{2\nu}\right)^{\frac{1}{\nu}} 
= (n+1)^\frac{1}{\nu}(s_\nu)^\frac{1}{\nu}(1+\varepsilon)^{2}.
\]
%\begin{align*} na_{max}(n) 
% &= n (a_{max}(n)^\nu)^{\frac{1}{\nu}} \\
%& \leq n \left( \sum \limits_{j=1}^n a_j(n)^\nu\right)^\frac{1}{\nu} \\
%&= n (s_\nu R(\nu, n))^\frac{1}{\nu}\cdot (n^{1-\nu})^\frac{1}{\nu} \\
%&\leq n^\frac{1}{\nu}(s_\nu)^\frac{1}{\nu} \left(  1 + r_\nu\frac{(1+\delta_n)^\nu}{n} \right)^\frac{1}{\nu}\\
%&= (s_\nu)^\frac{1}{\nu} \left(n + r_\nu(1+ \delta_n)^\nu\right)^{\frac{1}{\nu}}\\
%&\leq (s_\nu)^\frac{1}{\nu} \left(n(1+\varepsilon)^{2\nu} + (1+\varepsilon)^{2\nu}\right)^{\frac{1}{\nu}} \\
%&\leq (n+1)^\frac{1}{\nu}(s_\nu)^\frac{1}{\nu}(1+\varepsilon)^{2}.
%\end{align*} 
Setting $\nu=n$,  for $n\geq \max\{\nu^*(\varepsilon), n^*(\varepsilon)\}$, we
have 
\begin{equation*} na_{max}(n) \leq \sqrt[n]{n+1}\sqrt[n]{s_n}(1+\varepsilon)^2.
\end{equation*}
Since $s = \lim_{n \rightarrow \infty} \sqrt[n]{s_n}$ by
(A.1), taking  first the limit superior as $n \rightarrow \infty$  and
then as $\varepsilon \downarrow 0$, we see that 
% Hence, for any $\varepsilon > 0$,
\begin{equation} \limsup \limits_{n \rightarrow \infty}na_{max}(n)
  \leq  \lim_{ \varepsilon \downarrow 0} s(1+ \varepsilon)^2 =
  s.  \label{eq:up} \end{equation}

Next, for the lower bound for $na_{max}(n)$, 
we will make use of the fact that 
%\begin{equation*} 
$( n  a_{max}(n))^{\nu} \geq   n ^{\nu-1} \sum_{j=1}^{ n } a_j(n) ^{\nu}.$
 %\end{equation*}
Indeed, then for $\varepsilon > 0$,  by  \eqref{eq:e3},
\eqref{eq:e4} and \eqref{inequr},  for $\nu \geq
\nu^*(\varepsilon)$ and $n\geq n^*(\varepsilon)$, we have 
\[  na_{max}(n) \geq \left(s_{\nu}R(\nu, n)\right)^{\frac{1}{\nu}} 
\geq (s_\nu)^{\frac{1}{\nu}}\left(1-r_\nu\frac{(1+\delta_n)^\nu}{n}\right)^\frac{1}{\nu} 
\geq
(s_\nu)^{\frac{1}{\nu}}\left(1-\frac{(1+\varepsilon)^{2\nu}}{n}\right)^\frac{1}{\nu}. 
\]
%\begin{align*} na_{max}(n) &\geq \left(s_{\nu}R(\nu, n)\right)^{\frac{1}{\nu}} \\
%& \geq (s_\nu)^{\frac{1}{\nu}}\left(1-r_\nu\frac{(1+\delta_n)^\nu}{n}\right)^\frac{1}{\nu} \\
%&\geq (s_\nu)^{\frac{1}{\nu}}\left(1-\frac{(1+\varepsilon)^{2\nu}}{n}\right)^\frac{1}{\nu}  \\
%& =
%(s_\nu)^{\frac{1}{\nu}}\left(\left(1-\frac{(1+\varepsilon)^{2\nu}}{n}\right)^n\right)^\frac{1}{n\nu}.  \end{align*}
Taking limits as $n \rightarrow \infty$ and noting that 
$(1-\frac{(1+\varepsilon)^{2\nu}}{n})^n \sim \exp\{
  -(1+\varepsilon)^{2\nu }\}$ and $n\nu \rightarrow \infty$ as $n
\rightarrow \infty$, we obtain 
\begin{align*} \liminf \limits_{n \rightarrow \infty} na_{max}(n) \geq
(s_\nu)^{\frac{1}{\nu}}\liminf \limits_{n \rightarrow
  \infty}\left(\left(1-\frac{(1+\varepsilon)^{2\nu}}{n}\right)^n\right)^\frac{1}{n\nu}
\geq   (s_\nu)^\frac{1}{\nu}, 
\quad  \forall \nu \geq \nu^*(\varepsilon). \end{align*}
Sending $\nu \rightarrow \infty$ and recalling 
from (A.1) that $s = \lim_{\nu \rightarrow \infty} \sqrt[\nu]{s_\nu}$, we conclude that
\begin{equation}  \liminf \limits_{n \rightarrow \infty}na_{max}(n)  \geq s.
\label{eq:down} \end{equation}
Combining \eqref{eq:up} and \eqref{eq:down}, we see that the weights 
$\{a_j\}_{j \in \mathbb N}$ satisfy  (B.2), and thus Assumption B.  
\end{proof}

\section{Proof of Theorem \ref{MR}} \label{sec23}

We will prove a slightly stronger statement than Theorem \ref{MR},
namely we show in Section \ref{sec232}
that if
the first inequality in \eqref{eq:e1} is satisfied, then the lower bound 
\begin{equation}
\label{lower-bound}
\liminf_{n \to \infty}\frac{1}{b\left(n\right)n^r} \log \mathbb{P} \left(\bar{S}_n \geq x\right) \geq 
 -\left(\frac{x}{s}-\frac{s_1}{s}m\right)^r, \quad \forall x > s_1 m, 
\end{equation} 
holds; and  in Section \ref{sec233} we show that  the second
inequality in \eqref{eq:e1} implies the
upper bound 
\begin{equation}
\label{upper-bound}
\limsup_{n \to \infty}\frac{1}{b\left(n\right)n^r} \log \mathbb{P} \left(\bar{S}_n \geq x\right) \leq 
 -\left(\frac{x}{s}-\frac{s_1}{s}m\right)^r,  \quad \forall x > s_1 m.
\end{equation}
 First, in  Section
\ref{subs-prelim}, 
we summarize some relevant properties of slowly varying
functions. 
Throughout the section, the notation $f(x) \sim g(x)$ as $x
\rightarrow \infty$ for two functions $f, g: \mathbb R \rightarrow
\mathbb R$ means that $\lim \limits_{x \rightarrow \infty}
f(x)/g(x)=1$. Also, given a set $A$, $\ind_A$ will denote the
indicator function of $A$, which equals $1$ on $A$ and $0$ on the
complement.

\subsection{Properties of Slowly Varying Functions}
\label{subs-prelim}

We will need the following preliminaries on slowly varying
functions. Proposition 3 corresponds to Proposition 1.3.6 in \cite{9},
where Lemma 4 refers to (1.4) in \cite{4}. 

\begin{prop}[Properties of Slowly Varying Functions] \label{props}
Let $\ell: (0,  \infty) \rightarrow (0, \infty)$ be a slowly varying function (at infinity). Then \renewcommand{\labelenumi}{(\roman{enumi})}
\begin{enumerate}
\item $\lim \limits_{x \rightarrow \infty} \displaystyle \frac{\log \ell(x)}{\log x}=0$.
\item For any $\alpha \in \mathbb R$, the function $f(x) =
  \ell(x)^\alpha, x \in \mathbb R$, is slowly varying.
\item For any $\alpha > 0$, $x^\alpha l(x) \rightarrow \infty$ and $x^{-\alpha}l(x) \rightarrow 0$ as $x \rightarrow \infty$. 
\end{enumerate}
Furthermore, if $m: (0,  \infty) \rightarrow (0, \infty)$ is another slowly varying function then
\begin{enumerate} \setcounter{enumi}{3}
\item the functions $f(x)= \ell(x)m(x)$ and $g(x) = \ell(x)+m(x)$, $x \in
  \mathbb R$,  are slowly varying. 
\item if $m(x) \rightarrow \infty$ as $x \rightarrow \infty$, then the
  function $f(x) = \ell(m(x)), x \in \mathbb R,$ is slowly varying.
\end{enumerate}
\end{prop}

\begin{lem}[Representation for Slowly Varying Functions] \label{rsv2}
A function $\ell: (0, \infty) \rightarrow (0,\infty)$ is slowly
varying if and only if there exist $a>0$, $\bar{\eta} \in \mathbb{R}$
and bounded measurable functions $\eta(\cdot)$ and $\varepsilon
(\cdot)$ with $\eta(x) \rightarrow \bar{\eta}$, $\varepsilon (x) \rightarrow 0$ as $x \rightarrow \infty$ such that, for $x \geq a$, $\ell$ can be written in the form 
\begin{equation} \ell(x)=\exp \left\{\eta(x)+\int \limits_{a}^{x}
    \frac{\varepsilon(u)}{u}du \right\}. \label{rsv} \end{equation}
\end{lem}

As a direct consequence of Lemma \ref{rsv2}, we have the following
result. 

\begin{lem}\label{gsv}
Let $\ell: (0, \infty) \rightarrow (0,\infty)$ be a slowly varying
function and let $g: (0,\infty) \rightarrow (0, \infty)$ be another function such that $g(x) \rightarrow c$ for some $c \in (0, \infty)$ as $x \rightarrow \infty$. Then 
we have
\begin{equation} \lim \limits_{x \rightarrow \infty} \frac{\ell\left(g(x)x\right)}{\ell(x)}=1. \label{gsv2} \end{equation}
\end{lem}

\subsection{The Lower Bound} \label{sec232}

For $n \in \mathbb N$, let 
 $j^{*}(n):=\inf\{1\leq j \leq n: a_j(n)=a_{max}(n)\}$. 
For any fixed $\varepsilon > 0$,  since the $\{X_j\}_{j  \in
  \mathbb N}$ are i.i.d.,  
\begin{align*}
\mathbb{P}(\bar{S}_n\geq x) &= \mathbb{P}\left(\sum \limits_{j=1}^{ n } a_j(n)(X_j-m) \geq  x-\sum \limits_{j=1}^{ n }a_j(n)m\right) \\ &\geq 
\mathbb{P}\left(a_{\rm max}(n)(X_{j^*(n)}-m) \geq x-\sum \limits_{j=1}^{ n } a_j(n)m +\varepsilon, \sum \limits_{j\in \{1, \ldots ,n\}, j \neq j^*(n)}
 a_j(n)(X_j-m) \geq -  \varepsilon\right) \\
&=  \mathbb{P}\left(X_1 \geq t_1(n)\right)\mathbb{P}
\left(\sum\limits_{j\in \{1, \ldots ,n\}, j \neq j^*(n)} a_j(n)(X_j-m) \geq -
\varepsilon\right), 
\end{align*}
where $t_1(n) = t_1^\varepsilon (n)$ is defined by 
\begin{equation}
\label{def-t1n}
t_1(n) := \frac{1}{na_{\rm max} (n)}\left[n \left(x-\sum \limits_{j=1}^{ n
    }a_j(n)m + a_{\rm max} (n)m + \varepsilon\right) \right], \quad n \in
  \mathbb N. 
\end{equation}
Applying the lower bound of \eqref{eq:e1} with 
$t = t_1(n)$, we obtain 
\begin{equation}
\label{temp}
 \mathbb P\left(\bar{S}_n \geq x \right) \geq
  c_1\left(t_1(n)\right) \exp\left\{-b\left(t_1(n)\right)(t_1(n))^r\right\}\cdot 
\mathbb{P}\left( \sum\limits_{j\in \{1, \ldots ,n\}, j \neq j^*(n)}
  a_j(n)(X_j-m) \geq -  \varepsilon\right). 
\end{equation}
Note that by Assumption \ref{as_B}, $t_1(n)\sim
\left(\frac{x}{s}-\frac{s_1}{s}m
  +\frac{\varepsilon}{s}\right)n \label{asymp1}$ as $n \rightarrow
\infty$. Since $c_1(\cdot)$ and $b(\cdot)$ are slowly varying
functions,  Lemma \ref{gsv} implies  that
$c_1\left(t_1(n)\right) \sim c_1(n)$ and $b\left(t_1(n)\right) \sim
b(n)$ as $n \rightarrow \infty$. 
Moreover, note that  for some fixed $\delta \in (0,r)$, we can express 
\[ \log c_1(n)/b(n) n^r = (\log
c_1(n)/\log n) (\log n/n^\delta)(b(n) n^{r-\delta})^{-1},\] 
and the right-hand side goes  to zero as $n \rightarrow \infty$ by 
properties (i) and (iii) of  Proposition 
\ref{props}.  
Furthermore, since the $\{X_j\}$ have finite second
moments by \eqref{allmom}, and  
(B.2) implies that  
$\sum_{j=1, j \neq j^*(n)}^n a_j(n)^2 \leq n (a_{\max}(n))^2
\rightarrow 0$ as $n \rightarrow  \infty$, it follows  that $\sum_{j \in \{1, \ldots, n\}, j \neq j^*(n)} a_j(n) (X_j -
m)$ converges to $0$ in $\mathbb{L}^2$. 
 In turn, this implies that 
$\lim_{n \rightarrow \infty} \mathbb P ( \sum_{j\in \{1, \ldots ,n\}, j \neq j^*(n)}
a_j(n)(X_j-m) \geq -  \varepsilon) = 1.$ 
\iffalse
\[  \lim_{n \rightarrow \infty} \mathbb P \left( \sum \limits_{j\in \{1, \ldots ,n\}, j \neq j^*(n)}
a_j(n)(X_j-m) \geq -  \varepsilon\right) = 1. 
\]
\fi 
Thus, taking logarithms of both sides of (\ref{temp}), then dividing  by
$b(n)n^r$ and sending first $n\rightarrow \infty$, and then
$\varepsilon \downarrow 0$, 
we obtain the lower bound \eqref{lower-bound}. 

%\begin{equation}
%\liminf \limits_{n \rightarrow \infty}\frac{1}{b\left(n\right)n^r}\log
%\mathbb{P}\left( \sum \limits_{j\in \{1, \ldots n\}, j \neq i^*}a_j(n)(X_j-m) \geq -  \varepsilon\right)=0 \label{eq:lower}.
%\end{equation} 
%Now the claim follows as $\varepsilon \rightarrow 0$. 

\subsection{The Upper Bound} \label{sec233}
Let $t_2(n):=n\left(\frac{x}{s}-\frac{s_1}{s}m\right)$.  Then,  we can
write 
\begin{align} \mathbb{P}\left(\bar{S}_n \geq x\right) &\leq A_1^n +
  A_2^n, 
\label{eq:ub}
\end{align}
where, for $n \in \mathbb N$, 
%\begin{align*}
\[ 
A_1^n := 
\mathbb{P}\left(\max \limits_{1 \leq j \leq  n} X_j \geq t_2(n)
\right), \qquad 
A_2^n := \mathbb{P}\left(\bar{S}_n \geq x, \max \limits_{1 \leq j
     \leq  n} X_j < t_2(n) \right). 
\]
%  n\left(\frac{x}{s}-\frac{s_1}{s}m\right)\right), \qquad 
% A_2^n := \mathbb{P}\left(\bar{S}_n \geq x, \max \limits_{1 \leq j
%     \leq  n} X_j < n\left(\frac{x}{s}-\frac{s_1}{s}m\right)\right).
%\]
%\end{align*}
The union bound and the upper tail bound for $X_1$ in \eqref{eq:e1}  imply that  
\begin{equation*}
%\label{bd-a1n}
 A_1^n \leq n \mathbb P ( X_1 \geq t_2(n)) 
\leq n  c_2\left(t_2(n)\right) \cdot \exp\left\{-b\left(t_2(n)\right)
  (t_2(n))^r \right\}. 
\end{equation*}
%n ^r\left(\frac{x}{s}-\frac{s_1}{s}m\right)^r\right\}, \end{equation}
% where we set $t_2(n):=n\left(\frac{x}{s}-\frac{s_1}{s}m\right)$. 
Since $b$ is slowly varying, $b\left(t_2(n)\right) \sim
b\left(n\right)$ as $n \rightarrow \infty$, and properties (i) and
(iii) of  Proposition \ref{props} show that  $\lim_{n \rightarrow
  \infty} \log n /b(n) n^{r}$ $ = \lim_{n \rightarrow \infty}\log c_2(t_2(n))/b(n) n^r
= 0$.  Together with 
the last display, this  implies that 
\begin{equation} 
\label{bd-a1n}\limsup \limits_{n \rightarrow \infty}
  \frac{1}{b\left(n\right)n^r} \log A_1^n \leq \limsup_{n \rightarrow \infty}
  \dfrac{ -(t_2(n))^r}{n^r} = -
  \left(\frac{x}{s}-\frac{s_1}{s}m\right)^r. \end{equation}

Next, we turn to $A_2^n$. Applying the exponential Chebyshev
inequality 
%(which is justified  since all  $X_j$ are bounded on $A_2^n$)  
with a positive 
real parameter $\beta_\zeta (n)/s$ (to be
specified later) we obtain 
\begin{align} 
\label{restri} 
A_2^n 
%= \mathbb{P} \left(\sum \limits_{j=1}^{ n }
%    \frac{a_{j}(n)}{s}X_j \geq \frac{x}{s}, \max\limits_{1 \leq i \leq  n } X_j <
%    t_2(n) \right) \nonumber \\
%    n\left(\frac{x}{s}-\frac{s_1}{s}m\right)\right)\nonumber \\
& \leq \exp\left\{-\beta_\zeta(n)\frac{x}{s}\right\} \cdot 
\prod \limits_{j=1}^{ n }\mathbb{E}
\left[\exp\left\{\beta_\zeta(n) \frac{a_{j}(n)}{s}X_j\right\} \cdot \ind_{\left\{X_j  <t_2(n) \right\}} \right]. 
\end{align}
%      n\left(\frac{x}{s}-\frac{s_1}{s}m\right)\right\}}\right].
Now, for  $\zeta > 0$, 
 define \begin{equation}
\label{def-beta} \beta_\zeta (n):= \zeta n^r
   b\left(n\left(\frac{x}{s}-\frac{s_1}{s}m\right)\right) = \zeta n^r
   b( t_2(n)).\end{equation} 
Then, since $b(\cdot)$ is slowly varying, $\lim_{n \rightarrow 
  \infty}\beta_\zeta(n)/(b(n)n^r) = \zeta$.
Together with \eqref{restri} this implies that 
\begin{align}
\label{preprop}
\limsup_{n \rightarrow \infty} \frac{1}{b(n) n^r} \log A_2^n & \leq -
\zeta \frac{x}{s} + 
\limsup_{n \rightarrow \infty}\frac{1}{b(n) n^r} \sum_{j=1}^n \Lambda_\zeta^j(n), 
\end{align}
where, for $j = 1, \ldots, n$, $n \in \mathbb N$, and $\zeta > 0$,  we define  
\begin{equation}
\label{def-Lambda}  \Lambda_\zeta^j (n) := \log
\mathbb{E}\left[\exp\left\{\beta_\zeta(n)
    \frac{a_{j}(n)}{s}X_j^{(n)}\right\}\right], \quad \mbox{ where }
X_j^{(n)} :=X_j\ind_{\left\{X_j < t_2(n) \right\}}. 
\end{equation}
We now show that the upper bound \eqref{upper-bound} is satisfied if the following
proposition holds. 
\begin{prop}[Boundedness of the remainder] \label{remainder}
For every $\zeta <\left(\frac{x}{s}-\frac{s_1}{s}m\right)^{r-1}$, 
\begin{equation}
\label{eq-remainder} \limsup \limits_{n \rightarrow \infty}
  \frac{1}{b\left(n\right)n^r} \sum \limits_{j=1}^{ n }
    \Lambda^j_\zeta(n)  \leq \zeta m \frac{s_1}{s}.\end{equation}
 \end{prop} 
Indeed, given Proposition \ref{remainder}, we can  substitute
\eqref{eq-remainder} into \eqref{preprop} and send 
$\zeta\uparrow \left(\frac{x}{s}-\frac{s_1}{s}m\right)^{r-1}$ 
to conclude that 
\[  \limsup_{n \rightarrow \infty} \frac{1}{b(n) n^r} \log A_2^n  \leq -\left(\frac{x}{s} - \frac{s_1}{s}m \right)^r.  \]
Together with \eqref{eq:ub}, and the analogous bound \eqref{bd-a1n}
for $A_1^n$, we obtain the upper
bound \eqref{upper-bound}.

Thus, to prove the upper bound, it only remains to prove Proposition \ref{remainder}. We use similar techniques as in \cite{21}. 

\begin{proof}[Proof of Proposition \ref{remainder}]
%\textbf{Proof of Proposition \ref{remainder}.]
Fix $\zeta <
(\frac{x}{s}-\frac{s_1}{s}m)^{r-1}$ and denote 
$\beta_\zeta (n)$ and $\Lambda^j_\zeta$ simply as $\beta(n)$ and
$\Lambda^j$.  For the fixed $r \in (0,1)$, we also choose $k \in
\mathbb{N}$ such that $r < k/(k+1)$.
Then, by  the definition \eqref{def-Lambda} of $\Lambda^j$,  the estimates $\log x \leq
x-1$ for $x >0$ and $e^x - 1 \leq x + \frac{1}{2}x^2 + \frac{1}{6}x^3
+ ... + \frac{1}{(k+1)!}x^{k+1}e^x$, finiteness of the moments of
$X_j$ due to \eqref{allmom}, 
and the fact that $\beta(n)/(b(n)n^r) \rightarrow \zeta$  and $\sum_{j=1}^n a_j
\rightarrow 
s_1$ as $n \rightarrow \infty$, we have 
% $\mathbb{E}[X_1^n]< \infty$ for all $n \in \mathbb N$, we deduce that
\begin{eqnarray*}
\limsup \limits_{n \rightarrow \infty}
  \frac{1}{b\left(n\right)n^r}\sum \limits_{j=1}^{ n }
    \Lambda^j(n)  & \leq & \limsup \limits_{n \rightarrow \infty} 
  \frac{1}{b(n)n^r} \left( \sum \limits_{j=1}^n  \sum
    \limits_{i=1}^{k} \frac{\mathbb{E} \left[ \left(\beta(n)
      \frac{a_{j}(n)}{s}X_j^{(n)}\right)^i \right]}{i!}\right)  +
\frac{B_0}{(k+1)!}, 
\end{eqnarray*}
with  
\[  B_0 := \limsup \limits_{n \rightarrow
    \infty}\frac{1}{b(n)n^r} \sum
  \limits_{j=1}^n\left(  \beta(n) \frac{a_{j}(n)}{s}\right)^{k+1}
  \cdot \mathbb{E}\left[ \left(X_j^{(n)}\right)^{k+1} \exp\left(
      \beta(n) \frac{a_{j}(n)}{s}  X_j^{(n)} \right) \right]. \] 
Since, by Assumption \ref{as_B},  $\lim\limits_{n \rightarrow
  \infty} \frac{1}{b(n)n^r} \sum_{j=1}^n \mathbb{E}[
  (\beta(n) \frac{a_{j}(n)}{s}  X_j^{(n)} )^i] = \zeta m
  \frac{s_1}{s}$ if $i=1$, and is zero for  
% $i \in \mathbb N$, 
$i  \neq 1$, this implies 
\[ \limsup \limits_{n \rightarrow \infty}
  \frac{1}{b\left(n\right)n^r}\sum \limits_{j=1}^{ n }
    \Lambda^j(n)  \leq\zeta m
  \frac{s_1}{s} + \frac{B_0}{(k+1)!}. 
\]
%we used the fact that $\lim_{n \rightarrow 
%  \infty}\beta(n)/(b(n)n^r) = \zeta$  and $\sum_{j=1}^n a_j
%\rightarrow 
%s_1$ as $n \rightarrow \infty$, and 
%$B_0$ is the quantity 
To complete the proof of Proposition \ref{remainder}, it suffices to show that $B_0 = 0$. In this regard, we
distinguish between the cases $X_j^{(n)}<t^*$ and $X_j^{(n)} \geq
t^*$, where we recall that for $t \geq t^*$, \eqref{eq:e1} is
satisfied. Specifically, we bound $B_0$  by $\limsup\limits_{n \to \infty} (B_1(n) + B_2(n))$, where
\begin{align}
B_{1}(n)&:= \frac{1}{(k+1)!} \frac{1}{b(n)n^r}\sum
\limits_{j=1}^n\left( \beta(n) \frac{a_{j}(n)}{s}\right)^{k+1}  \cdot
\left(t^*\right)^{k+1} \exp\left( \beta(n) \frac{a_{j}(n)}{s}  t^*
\right) , \label{def-b1n}\\ 
B_{2}(n)&:=\frac{1}{(k+1)!} \frac{1}{b(n)n^r} \sum \limits_{j=1}^n\left(  \beta(n) \frac{a_{j}(n)}{s}\right)^{k+1}  \cdot \mathbb{E}\left[ \left(X_j^{(n)}\right)^{k+1} \exp\left( \beta(n) \frac{a_{j}(n)}{s}  X_j^{(n)} \right) \ind_{\left\{X_j^{(n)} \geq t^* \right\}}\right]. \label{def-b2n}\end{align}
We now show that both $B_1(n)$ and $B_2(n)$ converge to
$0$ as $n \rightarrow \infty$.  
%Recall that $a_{max}(n):=\max_{1 \leq j \leq n} a_j(n)$. 
%and 
Note that \ref{as_B2}, the definition of $\beta(n)$ in \eqref{def-beta} 
and, recalling $r < k/(k+1)$, property (iii) of Proposition
\ref{props} 
imply that
\begin{equation}
\label{justify}
\lim_{n \rightarrow \infty} n \left ( \beta(n) \frac{a_{\max}(n) }{s}
\right)^{k+1} = \lim_{n \rightarrow \infty}  \left(\frac{a_{\max}(n)
    n}{s}\right)^{k+1}  \left(\zeta n^{r - \frac{k}{k+1}} b(n)
\right)^{k+1} = 0, 
\end{equation}
 and 
\begin{equation}
\label{justify2} 
\lim_{n \rightarrow \infty} \left ( \beta(n) \frac{a_{\max}(n) }{s}
\right) = 0. 
\end{equation}
%We will make repeated use of the fact that for $\kappa \in
%(0,\infty)$, 
%by (\ref{eq:f4}), the definition of $\beta(n)$ in \eqref{def-beta} 
%and property (iii) of Proposition \ref{props}, 
%\begin{equation}
%\label{just}
%\frac{\kappa}{\kappa + 1} >
%r  \quad \Rightarrow \quad \lim_{n \rightarrow \infty} n \left ( \beta(n) \frac{a_{\max}(n) }{s}
%\right)^{\kappa+1} = \lim_{n \rightarrow \infty}  \left(\frac{a_{\max}(n)
%    n}{s}\right) \zeta^{k+1} \left(n^{r - \frac{k}{k+1}} b(n)
%\right)^{k+1} = 0. 
%\end{equation}
%A similar argument shows that $\left ( \beta(n) \frac{a_{\max}(n) }{s}
%\right) \rightarrow 0$ as $n \rightarrow \infty$.
Combined with \eqref{def-b1n} and  recalling that  $a_{max}(n):=\max_{1 \leq j \leq n} a_j(n)$,   this
shows that $B_1(n) \rightarrow 0$ as $n \rightarrow \infty$.  

Now, to bound $B_2(n)$, first note that by H\"{o}lder's inequality, 
 for any $\varepsilon >0$ we have 
\begin{align}
\label{holder} &\mathbb{E}\left[ \left(X_1^{(n)}\right)^{k+1} \exp\left( \beta(n) \frac{a_{max}(n)}{s}  X_1^{(n)} \right) \ind_{\{X_1^{(n)} \geq t^* \}} \right] \notag \\ &\leq \mathbb{E}\left[ \left(X_1^{(n)}\right)^{(k+1) \cdot \frac{1+\varepsilon}{\varepsilon}} \ind_{\{X_1^{(n)} \geq t^* \}}  \right]^{\frac{\varepsilon}{1+\epsilon}} \cdot \mathbb{E}\left[ \exp\left((1+\varepsilon) \beta(n) \frac{a_{max}(n)}{s}  X_1^{(n)} \right) \ind_{\{X_1^{(n)} \geq t^* \}}  \right]^{\frac{1}{1+\varepsilon}}.
\end{align}
Due to the finiteness of the moments of $X_1$ assumed in 
\eqref{allmom}, 
the limit in 
\eqref{justify} yields 
%and the elementary  estimate $a_j(n) \leq a_{max}(n)$ for $1 \leq j
%\leq n$, 
\[\limsup \limits_{n \rightarrow \infty} n\cdot\left(  \beta(n)
  \frac{a_{max}(n)}{s}\right)^{k+1} \mathbb{E}\left[
  \left(X_1^{(n)}\right)^{(k+1) \cdot
    \frac{1+\varepsilon}{\varepsilon}} \ind_{\{X_1^{(n)} \geq
  t^*\}}\right]^{\frac{\varepsilon}{1+\varepsilon}} =0. \]
When combined with \eqref{def-b2n} and \eqref{holder}, to prove the 
convergence of  
$B_2(n)$ to zero, it clearly suffices to show 
that 
\begin{equation} \limsup \limits_{n \rightarrow \infty} \frac{1}{b(n)n^r}   \mathbb{E}\left[ \exp\left((1+\varepsilon) \beta(n) \frac{a_{max}(n)}{s}  X_1^{(n)} \right)  \ind_{\{X_1^{(n)} \geq
  t^*\}} \right]^{\frac{1}{1+\varepsilon}} < \infty\label{pr5} \end{equation}
for $\zeta < (1+\varepsilon)^{-1}\left(\frac{x}{s}-\frac{s_1}{s}m\right)^{r-1}$ and the claim follows as $\varepsilon \rightarrow 0$.
To derive an upper bound for the expectation in \eqref{pr5}
we will use  the following integration-by-parts formula.

\begin{lem}[Integration by parts] \label{inteq}
For any random variable $X$ on a probability space $(\Omega, \mathcal F, \mathbb{P})$ and any $\alpha >0$, $a$, $b \in \mathbb{R}$ with $a  < b$ the following relation holds:
\begin{equation*} \mathbb{E}\left[\exp\left(\alpha X\right) \mathds{1}_{\left\{a \leq X \leq b\right\}}\right]=\alpha \int \limits_{a}^{b} \exp\left(\alpha z\right)\mathbb{P}\left(X\geq z\right)dz + \exp\left( \alpha a \right)\mathbb{P}\left(X \geq a\right) - \exp\left( \alpha b\right)\mathbb{P}\left(X > b\right). \end{equation*}
\end{lem}
Recalling that $X_j^{(n)} = X_j \ind_{\{X_j < t_2(n)\}}$, applying
  Lemma \ref{inteq} with $a = t^*$ and $b = t_2(n)$, we deduce that 

\begin{align} 
\frac{1}{b(n)n^r}  &\mathbb{E}\left[ \exp\left((1+\varepsilon) \beta(n) \frac{a_{max}(n)}{s}  X_1^{(n)} \right) \ind_{\{X_1^{(n)} \geq t^* \}} \right]\notag \\
&\leq \frac{1}{b(n)n^r} \int \limits_{t^*}^{t_2(n)} (1+\varepsilon) \beta(n) \frac{a_{max}(n)}{s} \exp \left( (1+\varepsilon) \beta(n) \frac{a_{max}(n)}{s}z\right) \mathbb{P}\left( X_1\geq z\right) dz \notag \\
& \quad + \frac{1}{b(n)n^r}  \exp\left((1+\varepsilon)\beta(n)\frac{a_{max}(n)}{s}t^*\right). \label{pr6}
\end{align}

%Recall that $a_{max}(n) \sim s\cdot n^{-1}$ as $n\rightarrow \infty$,
%and hence 
%As $n \rightarrow \infty$, 
Since $b(n) n^r
\rightarrow \infty$, 
the second term on the right-hand side of \eqref{pr6} converges to $0$
by \eqref{justify2}.  
Now, let $\zeta^{*}:=\zeta\cdot
\left(\frac{x}{s}-\frac{s_1}{s}m\right)$. 
% and note that then $(1+\varepsilon)^{-1}\left(\frac{x}{s}-\frac{s_1}{s}m\right)^r$] 
 Inserting the upper bound \eqref{eq:e1}
on the tail of $X_1$, substituting $y:= (t_2(n))^{-1} z$
and recalling the definition of $\beta(n)$ from \eqref{def-beta}, we see that
the first term on the right-hand side of  \eqref{pr6} is bounded above by 
\begin{align} 
&(1+\varepsilon)\zeta^{*} \frac{b(t_2(n))}{b(n)} \frac{na_{max}(n)}{s} \quad
\cdot \int
\limits_{\frac{t^*}{t_2(n)}}^{^1}
I_n( y) dy, \label{pr7} \end{align}
where the integrand $I_n(\cdot)$ is given by \begin{align*}
   I_n( y)&:=c_2
   \left( t_2(n) y\right) \exp \left\{  n^r b\left( t_2(n)\right)
 \left((1+\varepsilon)\zeta^{*} \frac{na_{max}(n)}{s} y -
   \frac{b(t_2(n)y)}{b(
     t_2(n))}\left(\frac{x}{s}-\frac{s_1}{s}m\right)^r
   y^r\right)\right\}, \end{align*}
for $y \in (0,1].$ 
Since $b(\cdot)$ is slowly varying and condition (B.2) holds, we see
that the coefficient in front of the integral in \eqref{pr7} converges
to $(1+\varepsilon) \zeta^{*}$ as $n \rightarrow \infty$.   It now
remains to show that, for every
$\zeta^{*}<(1+\varepsilon)^{-1}\left(\frac{x}{s}-\frac{s_1}{s}m\right)^r$,
the integral in \eqref{pr7} stays bounded as $n\rightarrow \infty$.
% due to assumption $\eqref{eq:e1}$. More precisely,
 By the assumption that $b(\cdot)$ is slowly varying and since $r <
 1$, 
 for any fixed $y \in (0,1]$ and any $\zeta^{*} < (1+\varepsilon)^{-1}
 \left(\frac{x}{s}-\frac{s_1}{s}m\right)^{r}$, it follows that  $I_n( y) \rightarrow 0$
 as $n \rightarrow \infty$. Therefore, we need to examine the lower
 limit of integration
 $y_n:=t^{*}/(t_2(n))$ 
and show that  $I_n( y_n)$ stays bounded as $n \rightarrow
\infty$.   Recalling that  $t_2(n) = n (\frac{x}{s} - \frac{s_1}{s}
m)$ and $\zeta^* = \zeta (\frac{x}{s} - \frac{s_1}{s}
m)$, note that
\[  I_n (y_n)  =   c_2(t^*) \exp \left\{ n^{r-1} b(t_2(n))
  (1+\varepsilon)
  \zeta \frac{n a_{\max} (n)}{s} t^*- b(t^*) (t^*)^r \right\}. 
\]
Since $na_{max}(n) \sim s$,   
% by (B.2),
$b(t_2(n)) 
  \sim b(n)$ and $n^{r-1} b(n) \rightarrow 0$ as $n \rightarrow
\infty$, 
% by property (iii) of Proposition \ref{props},
 it follows that $\limsup_{n \rightarrow
   \infty} I_n(y_n)$ is finite.

\iffalse
Since $a_{max}(n) \sim sn^{-1}$, $b(t_2(n))   \sim b(n)$ as $n
\rightarrow \infty$, 
 we can find a constant $L >0$ such that, for $n$ large enough and $\zeta^{*} < (1+\varepsilon)^{-1} \left(\frac{x}{s}-\frac{s_1}{s}m\right)^{r}$,
\begin{equation*}
I_n(y_n) \leq c_2(t^*)\cdot \exp \left\{ K
  \left(\frac{x}{s}-\frac{s_1}{s}m\right)^{r-1}t^{*}b(n)/n^{1-r}-b(t^*){t^{*}}^r\right\}.\end{equation*}
 By property (iii) of Proposition \ref{props},  $b(n)/n^{1-r} \to 0$
 for $n \to \infty$, and hence, $\limsup_{n \rightarrow
   \infty} I_n(y_n) < \infty$. 
\fi 
Thus, we have shown that $B_2^n$ converges to zero as $n \rightarrow
\infty$ and hence, that $B_0 = 0$.  This  completes the proof of Proposition
\ref{remainder}, and hence, the upper bound \eqref{upper-bound} and
Theorem \ref{MR} follow. 
\end{proof}

\section{Examples} \label{Examples}
\subsection{Example 1: Random Weights} \label{Example1}

We consider a sequence of strictly positive i.i.d. random variables 
$\{\theta_j\}_{j \in \mathbb N}$ on $(\Omega, \mathcal F, \mathbb P)$ 
%with $\mathbb{P}\left( \theta_1 > 0\right) =1$, 
and assume that they are $\mathbb{P}$-almost surely uniformly bounded,
that is, their essential supremum is finite: 
\begin{equation} M^{*}:= \inf \left\{a \in \mathbb R: \mathbb P \left( \theta_1 > a\right)=0\right\} < \infty. \label{rwem2} \end{equation}
%Note that, by \eqref{rwem1}, it immediately follows that for the distribution of 
%$\theta_1$ all moments exist, i.e. 
%for every $\nu \in \mathbb N$ we have $ \mathbb{E}[\theta_1^\nu]< \infty$. 
%Define the sequence of weights $\left(a_j(n, \theta_j)\right)$by
%\begin{equation} a_j(n, \theta_j):=\begin{cases} \frac{\theta_j}{n}, &1\leq j \leq n \\ 0, &\text{else} 
%\end{cases}. \label{rwem3} \end{equation}
%and consider the randomly weighted empirical mean 
%\begin{equation} Q(n):=\sum \limits_{j=1}^{n} a_j(n,\theta_j)X_j=\frac{1}{n}\sum \limits_{j=1}^{n} 
%\theta_jX_j. \label{rwem4} \end{equation}

Furthermore, define the triangular array  of weights $\left\{a_j(n,
  \theta_1,..., \theta_n), j = 1, \ldots, n\right\}_{n \in \mathbb{N}}$ by
\begin{equation} a_j(n, \theta_1, ..., \theta_n):=\frac{\theta_j}{\sum
    \limits_{i=1}^{n}\theta_i}, \quad j = 1, \ldots, n,  n \in \mathbb N,
\label{rp3} \end{equation}
and let $\{\bar{S}_n\}_{n \in \mathbb N}$ be the corresponding sequence
of weighted sums: 
\begin{equation} \bar{S}_n:=\sum \limits_{j=1}^{n}
  a_j(n,\theta_1,...,\theta_n)X_j=\sum \limits_{j=1}^{n}
  \frac{\theta_j}{\sum \limits_{i=1}^{n}\theta_i}X_j, \quad n \in \mathbb N. \label{rp4} \end{equation}
We prove a large deviation theorem for the
sequence of random weighted sums $\{ \bar{S}_n\}_{n \in \mathbb{N}}$, 
both in the ``quenched''  (i.e., conditioned on the weight
sequence $\{\theta_j\}_{j \in \mathbb N}$), and ``annealed'' (i.e., averaged
over the weight sequence) cases.    
Note that $\bar{S}_n$ can be
viewed as a random projection of the data $\{X_i\}$.  
Random projections have attracted much interest in
recent research in applied mathematics as an important tool in
data analysis and dimensionality reduction \cite{BinMan01}, as well as
in asymptotic geometric analysis \cite{DiaFre84, Mec12}. 

\begin{thm}[Large Deviations for Random Weights, Stretched Exponential Tails] \label{quran44}
Let $\{ X_j \}_{j \in \mathbb{N}}$ be a sequence of i.i.d. random variables such
as in Theorem \ref{MR} and let $\{\theta_j\}_{j \in \mathbb{N}}$ be a sequence of
i.i.d. random variables which is independent
of the sequence $\{X_j\}_{j \in \mathbb{N}}$, and is  almost surely
uniformly bounded by $M^*$ as specified in \eqref{rwem2}. Define $\bar{S}_n$ by \eqref{rp4}.
Then, for $x> m$, we have 
\begin{equation}\label{quenched} \lim \limits_{n \rightarrow \infty} \frac{1}{b\left(n\right)n^r}\log \mathbb P\left(\bar{S}_n \geq x\right|\theta_1, \theta_2,...) =-\left[\left(\frac{\mathbb{E}[\theta_1]}{M^{*}}\right) \left(x-m\right) \right]^r \quad \mathbb P \text{-a.s.},\end{equation}
and
\begin{equation}\label{annealed} \lim \limits_{n \rightarrow \infty} \frac{1}{b\left(n\right)n^r}\log \mathbb P\left(\bar{S}_n \geq x\right) =-\left[\left(\frac{\mathbb{E}[\theta_1]}{M^{*}}\right) \left(x-m\right) \right]^r. \end{equation}
\end{thm}

\begin{proof} The proof of \eqref{quenched} is a direct application of Theorem \ref{MR}. First of all, note that for every $n \in \mathbb N$, $\sum_{j=1}^n a_j(n,\theta_1,...,\theta_n)=1$ almost surely,
  and hence $s_1=1$, where $s_1$  is the quantity defined in \ref{as_B1}.  Furthermore,
\begin{equation} n\cdot a_{max}(n, \theta_1,..., \theta_n)=\frac{n\cdot \max\{\theta_j: 1 \leq j\leq n \}}{\sum \limits_{i=1}^{n}\theta_i} = \frac{\max\{\theta_j: 1 \leq j\leq n \}}{\frac{1}{n}\sum \limits_{i=1}^{n} \theta_i}. \end{equation}
It is easy to check that almost surely,  $\max\{\theta_j: 1 \leq j\leq n \}
\rightarrow M^{*}$ as $n\rightarrow \infty$. By the strong law of
large numbers, it follows that almost surely, $n\cdot a_{max}(n,
\theta_1,...,\theta_n) \rightarrow s:= M^{*}/\mathbb{E}[\theta_1]$  as
$n \rightarrow \infty$. By Theorem \ref{MR} we conclude that, for $x
>m$, the quenched asymptotics \eqref{quenched} are valid.  

We now turn to the proof of \eqref{annealed}. 
Note that we have 
\begin{equation} 
\mathbb{P}\left(\bar{S}_n \geq x\right) = \mathbb{P}\left(\frac{\frac {1}{n} \sum\limits_{j=1}^{n} \theta_j X_j}{\frac{1}{n}\sum\limits_{i=1}^{n}\theta_i} \geq x\right).
\end{equation}
Now, $\frac{1}{n}\sum_{i=1}^{n}\theta_i \to
\mathbb{E}[\theta_1]$, $\mathbb{P}$-almost surely, and the probability
of a deviation decays exponentially in $n$, due to Cram\'er's Theorem
(recall that the $\{\theta_i\}$ are uniformly bounded!).
We will now show that 
\begin{equation}\label{simpel}
\lim \limits_{n \rightarrow \infty} \frac{1}{b\left(n\right)n^r}\log \mathbb P\left(\bar{S}_n \geq x\right)
\approx \lim \limits_{n \rightarrow \infty} \frac{1}{b\left(n\right)n^r}\log \mathbb P\left(\frac{1}{n}\sum \limits_{j=1}^{n} 
\theta_j X_j \geq \mathbb{E}[\theta_1] x\right),
\end{equation}
in the sense explained in  \eqref{ub} and \eqref{lb} below.  
Fix $\delta > 0$ and consider the events $F_n: = \{\frac{1}{n}\sum_{i=1}^{n}\theta_i \geq (1-\delta)\mathbb{E}[\theta_1]\}$ 
and
their complements $F_n^c$ for $ n \in \mathbb{N}$.  Then,
$\mathbb P\left(\bar{S}_n \geq x\right) \leq  \mathbb P(\frac{1}{n}\sum_{j=1}^{n} 
\theta_j X_j \geq (1-\delta) \mathbb{E}[\theta_1]x ) + \mathbb
P(F_n^c)$, and since $\mathbb P(F_n^c)$  decays exponentially in $n$,
it follows that for any $\delta > 0$, 
\begin{equation}\label{ub}
\limsup_{n \rightarrow \infty} \frac{1}{b\left(n\right)n^r}\log \mathbb P(\bar{S}_n \geq x)
\leq \limsup_{n \rightarrow \infty}  \frac{1}{b\left(n\right)n^r}\log \mathbb P\left(\frac{1}{n}\sum \limits_{j=1}^{n} 
\theta_j X_j \geq (1-\delta)\mathbb{E}[\theta_1] x\right).
\end{equation}
 On the other hand, with 
$G_n: = \{\frac{1}{n}\sum_{i=1}^{n}\theta_i \leq (1+\delta)\mathbb{E}[\theta_1]\}$, we have
$\mathbb P(\bar{S}_n \geq x) \geq  \mathbb P(\{\bar{S}_n \geq x\} \cap G_n) \geq  \mathbb P(\frac{1}{n}\sum_{j=1}^{n} 
\theta_j X_j \geq (1+ \delta) \mathbb{E}[\theta_1]x) -  \mathbb
P(G_n^c)$, and since  $\mathbb P(G_n^c)$ decays exponentially in $n$,
we have 
\begin{equation}\label{lb}
\liminf_{n \rightarrow \infty} \frac{1}{b\left(n\right)n^r}\log \mathbb P(\bar{S}_n \geq x)
\geq \liminf_{n \rightarrow \infty}  \frac{1}{b\left(n\right)n^r}\log \mathbb P\left(\frac{1}{n}\sum \limits_{j=1}^{n} 
\theta_j X_j \geq (1+ \delta)\mathbb{E}[\theta_1] x\right). 
\end{equation}
 Looking at the right-hand sides of \eqref{ub} and \eqref{lb} 
%and \eq{\eqref{simpel}, 
we are in the situation of Theorem \ref{MR} with i.i.d. random variables 
$\theta_j X_j$ and weights $a_j(n) = \frac{1}{n}, j = 1, \ldots ,n$
that clearly satisfy Assumption B with $s = s_1 = 1$ and $R(\nu, 1) = 1$
for all $\nu \in \mathbb{N}$.  Considering the tail of $\theta_1 X_1$, we see that 
due to \eqref{eq:e1}, for $t\geq t^*$,
$\mathbb P(\theta_1 X_1 \geq t) \leq  \mathbb P(X_1 \geq t/M^*) \leq c_2(t/M^*) \exp(-b(t/M^*) t^r (M^*)^{-r})$.
On the other hand, for $t\geq t^*$, again by \eqref{eq:e1}, $\mathbb P(\theta_1 X_1 \geq t) \geq \mathbb P(\theta_1 \geq M^* -\delta)
\mathbb P(X_1 \geq t/(M^* - \delta)) \geq  \mathbb P(\theta_1 \geq M^* -\delta)c_1(t/(M^* -\delta))\exp(-b(t/(M^* -\delta))t^r(M^* -\delta)^{-r})$.
The proof is completed by applying the lower and upper bounds in \eqref{lower-bound}
and \eqref{upper-bound}, respectively, and then sending $\delta
\downarrow 0$ to obtain \eqref{annealed}. 
\end{proof}

\begin{rem}{\em 
The equality of the quenched and annealed rate
functions in \eqref{quenched}  and \eqref{annealed}, respectively, is
characteristic of our regime; 
it is in sharp contrast to the case of light-tailed random variables
$X_j$, that is, random variables $X_j$ satisfying
\eqref{expmom}.   In the light-tailed case,
$\mathbb P\left(\bar{S}_n \geq x\right|\theta_1, \theta_2,...)$ and $\mathbb P\left(\bar{S}_n \geq x\right)$ both decay exponentially in $n$, but the rate functions will in general not be the same.
This was one of the motivations for the present paper, and will be treated in forthcoming work.}
\end{rem}

\subsection{Example 2: Kernel Functions} \label{Example2}

In non-parametric regression kernels are frequently used as weighting functions. They are an important tool to smooth data.  Applications include the approximation of probability density functions and conditional expectations.

\begin{defn}[Kernel] \label{kernels}
A kernel is an integrable function $k: [-1,1] \rightarrow [0, \infty)$ satisfying the following two requirements:
\renewcommand{\labelenumi}{(\roman{enumi})}
\begin{enumerate}
   \item $\int \limits_{-1}^1 k(u)du=1$.
   \item $k(-u)=k(u) \quad \forall u \in [0,1 ] $. 
\end{enumerate}
\end{defn}

Define the triangular array of weights $\left\{a_j(n), j = 1, \ldots, n\right\}_{n \in \mathbb{N}}$ by
\begin{equation} a_j(n):=\frac{1}{n}\cdot k\left(2 \cdot \frac{j-n/2}{n}\right), \quad j = 1, \ldots, n,  n \in \mathbb N,
\label{ker3} \end{equation}
and let $\{\bar{S}_n\}_{n \in \mathbb N}$ be the corresponding sequence
of weighted sums: 
\begin{equation} \bar{S}_n:=\sum \limits_{j=1}^{n}
  a_j(n)X_j=\frac{1}{n}\sum \limits_{j=1}^{n}
  k\left(2 \cdot \frac{j-n/2}{n}\right)X_j, \quad n \in \mathbb N. \label{ker4} \end{equation}

\begin{thm}[Large Deviations for Kernel Weighted Sums, Stretched Exponential Tails]
Let $\{ X_j \}_{j \in \mathbb{N}}$ be a sequence of i.i.d. random variables such
as in Theorem \ref{MR} and let $k: [-1,1] \rightarrow [0, \infty)$ be a kernel.
Define $\bar{S}_n$ by \eqref{ker4}.
Then, for $x> m$, we have 
\begin{equation} \lim \limits_{n \rightarrow \infty} \frac{1}{b\left(n\right)n^r} \log \mathbb P \left( \bar{S}_n \geq x \right) = -\left(\sup \limits_{x \in [-1,1]} k(x)\right)^{-r}\left({x}-m\right)^r. \end{equation}\end{thm}

\begin{proof} The proof is a direct application of Theorem \ref{MR}. Recall the definition of the quantities $\{s_\nu\}_{\nu \in \mathbb N}$ from Assumption \ref{as_B}. It is straightforward to check that $s_\nu = \int \limits_{-1}^1 k^{\nu}(u) du$ (in particular $s_1=1$). Therefore,
\begin{equation*} s = \lim \limits_{\nu \rightarrow \infty} \left({\int \limits_{-1}^1 k^\nu(u) du}\right)^{1/\nu}. \end{equation*}
Since the $p$-norm converges to the supremum norm as $p \rightarrow \infty$, we conclude that $s= \sup \limits_{x \in [-1,1]} k(x)$.\end{proof}

{\bf Acknowledgments. } N. Gantert and F. Rembart thank the Division of Applied
  Mathematics, Brown University, Providence, for its
  hospitality. N. Gantert further thanks ICERM, Providence, for an
  invitation to the program ``Computational Challenges in
  Probability'' where this work was initiated.

\end{document}